\documentclass[12pt]{article}
\usepackage{amsmath,amssymb,amsthm} 
\usepackage[unicode,breaklinks=true,colorlinks=true]{hyperref}
\usepackage[dvipsnames]{xcolor}
\numberwithin{equation}{section}
\newtheorem{theorem}{Theorem}[section]
\newtheorem{proposition}[theorem]{Proposition}
\newtheorem{lemma}[theorem]{Lemma}
\newtheorem{definition}[theorem]{Definition}

\theoremstyle{remark}
\newtheorem{remark}[theorem]{Remark}

\newcommand{\bke}[1]{\left( #1 \right)}

\newcommand{\bket}[1]{\left\{ #1 \right\}}
\newcommand{\norm}[1]{\| #1 \|}

\newcommand{\al}{\alpha}

\newcommand{\de}{\delta}
\newcommand{\e}{\epsilon}
\newcommand{\ga}{{\gamma}}

\newcommand{\la}{\lambda}

\newcommand{\Om}{{\Omega}}

\newcommand{\td}{\tilde}

\newcommand{\Bp}{\dot B_{p,\infty}^{3/p-1}}

\newcommand{\R}{{\mathbb R }}\newcommand{\RR}{{\mathbb R }}
\newcommand{\N}{{\mathbb N}}

\newcommand{\Z}{{\mathbb Z}}

\newcommand{\pd}{{\partial}}

\newcommand{\lec}{\lesssim}

\newcommand{\I}{\infty}
 
\renewcommand{\div}{\mathop{\mathrm{div}}}

\newcommand{\supp}{\mathop{\mathrm{supp}}}

\newcommand{\donothing}[1]{{}}
\newcommand{\EQ}[1]{\begin{equation}\begin{split} #1 \end{split}\end{equation}}
\newcommand{\EQN}[1]{\begin{equation*}\begin{split} #1 \end{split}\end{equation*}}

\DeclareMathOperator*{\esssup}{ess\,sup}

\makeatletter
\newcommand{\xRightarrow}[2][]{\ext@arrow 0359\Rightarrowfill@{#1}{#2}}
\makeatother

\newcommand{\loc}{\mathrm{loc}} 
\newcommand{\uloc}{\mathrm{uloc}}

\usepackage{accents}

\begin{document}

\title{Discretely self-similar solutions to the Navier-Stokes equations with data in $L^2_{\loc}$ satisfying the local energy inequality} 
\author{Zachary Bradshaw and Tai-Peng Tsai}
\date{\today}
\maketitle 

\begin{abstract}Chae and Wolf recently constructed discretely self-similar solutions to the Navier-Stokes equations for any discretely self similar data in $L^2_{\loc}$.  Their solutions are in the class of local Leray solutions with projected pressure, {and satisfy the ``local energy inequality with projected pressure''.}  In this note, {for the same class of initial data,} we construct discretely self-similar suitable weak solutions to the Navier-Stokes equations 
 {that satisfy the classical local energy inequality of Scheffer and Caffarelli-Kohn-Nirenberg. We also}
obtain an explicit formula for the pressure in terms of the velocity.  
Our argument involves a new purely local energy estimate for discretely self-similar solutions with data in $L^2_{\loc}$ and an approximation of divergence free, discretely self-similar vector fields in $L^2_{\loc}$ by divergence free, discretely self-similar elements of $L^3_w$.
\end{abstract}


\section{Introduction}

The Navier-Stokes equations describe the evolution of a viscous incompressible fluid's velocity field $v$ and associated scalar pressure $\pi$.  In particular, $v$ and $\pi$ are required to satisfy
\begin{align}\label{eq.NSE}
 &\partial_tv-\Delta v +v\cdot\nabla v+\nabla \pi = 0,
\\& \nabla \cdot v =0,
\end{align}
in the sense of distributions. 
For our purposes, \eqref{eq.NSE} is applied on $\R^3\times (0,\I)$ and $v$ evolves from a prescribed, divergence free initial data $v_0:\R^3\to \R^3$.  Solutions to \eqref{eq.NSE} exhibit a 
natural scaling: if $v$ satisfies \eqref{eq.NSE}, then for any $\lambda>0$
\begin{equation}
	v^{\lambda}(x,t)=\lambda v(\lambda x,\lambda^2t),
\end{equation}
is also a solution with pressure 
\begin{equation}
	\pi^{\lambda}(x,t)=\lambda^2 \pi(\lambda x,\lambda^2t),
\end{equation}
and initial data 
\begin{equation}
v_0^{\lambda}(x)=\lambda v_0(\lambda x).
\end{equation}
A solution is called self-similar (SS) if $v^\lambda(x,t)=v(x,t)$ for all $\lambda>0$ and is discretely self-similar with factor $\lambda$ (i.e.~$v$ is $\lambda$-DSS) if this scaling invariance holds for a given $\lambda>1$. Similarly, $v_0$ is self-similar (a.k.a.~$(-1)$-homogeneous) if $v_0(x)=\lambda v_0(\lambda x)$ for all $\lambda>0$ or $\lambda$-DSS if this holds for a given $\lambda>1$.  
These solutions can be either forward or backward if they are defined on $\R^3\times (0,\I)$ or $\R^3\times (-\I,0)$ respectively.  In this note we work exclusively with forward solutions and omit the qualifier ``forward''.

Self-similar solutions satisfy an ansatz for $v$ in terms of a time-independent profile $u$, namely, 
\begin{equation}\label{ansatz1}
v(x,t) = \frac 1 {\sqrt {t}}\,u\bigg(\frac x {\sqrt{t}}\bigg), 
\end{equation} 
where $u$ solves the \emph{Leray equations}
\begin{equation} 
\begin{array}{ll}\label{eq:stationaryLeray}
 -\Delta u-\frac 1 2 u-\frac 1 2 y\cdot\nabla u +u\cdot \nabla u +\nabla p = 0&
\\[3pt]  \nabla\cdot u=0&
\end{array}
\mbox{~in~}\R^3,
\end{equation}
in the variable $y=x/\sqrt{ t}$.
Discretely self-similar solutions are determined by their behavior on the time interval $1\leq t\leq \lambda^2$ and satisfy the ansatz
\begin{equation}\label{ansatz2}
v(x,t)=\frac 1 {\sqrt{t}}\, u(y,s),
\end{equation}
where
\begin{equation}\label{variables}
y=\frac x {\sqrt{t}},\quad s=\log t.
\end{equation}
The vector field $u$ is $T$-periodic with period $T=2\log \lambda$ and solves the \emph{time-dependent Leray equations}
\begin{equation} 
\begin{array}{ll}
\label{eq:timeDependentLeray}
 \partial_s u-\Delta u-\frac 1 2 u-\frac 1 2 y\cdot\nabla u +u\cdot \nabla u +\nabla p = 0& 
\\[3pt]  \nabla\cdot u = 0&
\end{array}
\mbox{~in~}\R^3\times \R.
\end{equation}
Note that the \emph{similarity transform} \eqref{ansatz2}--\eqref{variables} gives a one-to-one correspondence between solutions to \eqref{eq.NSE} and \eqref{eq:timeDependentLeray}. Moreover, when $v_0$ is SS or DSS, the initial condition $v|_{t=0}=v_0$ corresponds to a boundary condition for $u$ at spatial infinity, see \cite{KT-SSHS,BT1,BT2}.  

Self-similar solutions are interesting in a variety of contexts as candidates for ill-posedness  or finite time blow-up of solutions to the 3D Navier-Stokes equations (see \cite{GuiSve,JiaSverak,JiaSverak2,leray,NRS,Tsai-ARMA} and the discussion in \cite{BT1}. Forward self-similar solutions are compelling candidates for non-uniqueness \cite{JiaSverak2,GuiSve}. Until recently, the existence of forward self-similar solutions was only known for small data \cite{Barraza,CP,GiMi,Koch-Tataru,Kato}. Such solutions are necessarily unique.  In \cite{JiaSverak},  Jia and \v Sver\'ak constructed forward self-similar solutions for large data where the data is assumed to be H\"older continuous away from the origin. 
This result has been generalized in a number of directions by a variety of authors \cite{BT1,BT2,BT3,Chae-Wolf,KT-SSHS,LR2,Tsai-DSSI}.  This paper can be understood in the context of \cite{BT1,Chae-Wolf,LR2} and we briefly recall the main results of these papers.

In \cite{BT1}, we generalize \cite{JiaSverak} in two ways.  First, all smoothness assumptions on the initial data are removed; we only require $v_0\in L^3_w$ (and $v_0$ divergence free and SS or DSS).  Second, we allow the data to be DSS for any $\la>1$, in which case we obtain DSS solutions as opposed to SS solutions -- in contrast, the method of \cite{JiaSverak} can be adapted to give DSS solutions but only when $\la$ is close to $1$ \cite{Tsai-DSSI}.  The method of proof in \cite{BT1} has since been extended to the half-space in \cite{BT2} and to initial data in the Besov spaces $\Bp$ when $3<p<6$ \cite{BT3}. Solutions which satisfy 
a rotationally corrected scaling invariance
are also constructed in \cite{BT2}.

The solutions of \cite{BT1} belong to the class of \emph{local Leray solutions}.  This class was introduced by Lemari\'e-Rieusset  in \cite{LR} to provide a local analogue of Leray's weak solutions \cite{leray}.  We recall the definition of local Leray solutions in full.  For $q \in [1,\infty)$, we say $f\in L^q_{\uloc}$ if 
\[
\norm{f}_{L^q_{\uloc}} =\sup_{x \in\R^3} \norm{f}_{L^q(B(x,1))}<\infty.
\]

\begin{definition}[Local Leray solutions]\label{def:localLeray} A vector field $v\in L^2_{\loc}(\R^3\times [0,\infty))$ is a local Leray solution to \eqref{eq.NSE} with divergence free initial data $v_0\in L^2_{\uloc}$ if:
\begin{enumerate}
\item for some $\pi\in L^{3/2}_{\loc}(\R^3\times [0,\infty))$, the pair $(v,\pi)$ is a distributional solution to \eqref{eq.NSE},
\item for any $R>0$, $v$ satisfies
\begin{equation}\notag
\esssup_{0\leq t<R^2}\,\sup_{x_0\in \R^3}\, \int_{B_R(x_0 )}\frac 1 2 |v(x,t)|^2\,dx + \sup_{x_0\in \R^3}\int_0^{R^2}\int_{B_R(x_0)} |\nabla v(x,t)|^2\,dx \,dt<\infty,\end{equation}
\item for all compact subsets $K$ of $\R^3$ we have $v(t)\to v_0$ in $L^2(K)$ as $t\to 0^+$,
\item $v$ is suitable in the sense of Caffarelli-Kohn-Nirenberg, i.e., for all cylinders $Q$ compactly supported in  $ \R^3\times(0,\infty )$ and all non-negative $\phi\in C_0^\infty (Q)$, we have 
\EQ{\label{CKN-LEI}
&\int |v(t)|^2\phi \,dx +2\int \int |\nabla v|^2\phi\,dx\,dt 
\\&\leq 
\int\int |v|^2(\partial_t \phi + \Delta\phi )\,dx\,dt +\int\int (|v|^2+2\pi)(v\cdot \nabla\phi)\,dx\,dt,
}
\item for every $x_0\in \R^3$, there exists $c_{x_0}\in L^{3/2}(0,T)$ such that
\begin{align*}
		p(x,t)-c_{x_0}(t)&=-\frac 1 3 |v(x,t)|^2 +\frac 1 {4\pi}\int_{B_2(x_0)} K(x-y):v(y,t)\otimes v(y,t)\,dy
		\\&+\frac 1 {4\pi} \int_{\R^3\setminus B_2(x_0)} (K(x-y)-K(x_0-y)):v(y,t)\otimes v(y,t)\,dy,
\end{align*}
{in $L^{3/2}(0,T; L^{3/2}(B_1(x_0)))$,}
where $K(x)=\nabla^2(1/|x|)$.
\end{enumerate}
\end{definition}

In \cite{LR}, Lemari\'e-Rieusset constructed global in time local Leray solutions if $v_0$ belongs to $E^2$, the closure of $C_0^\infty$ in the $L^2_{\uloc}(\R^3)$ norm.  
See Kikuchi-Seregin \cite{KiSe} for another construction which treats the pressure carefully.    Note that \cite{LR}, \cite{KiSe} and \cite{JiaSverak,JiaSverak2} contain alternative definitions of local Leray solutions.  On one hand, \cite{KiSe} requires the pressure satisfies a certain formula (we will {establish a similar} pressure formula for our solutions, see Theorem \ref{thrm.main}).  In \cite{JiaSverak,JiaSverak2}, the explicit pressure formula is replaced by a decay condition imposed on the solution at spatial infinity, namely, for all $R>0$
\[
\lim_{|x_0|\to \I} \int_0^{R^2}\int_{{B(x_0,R)}} |v|^2\,dx\,dt = 0.
\]
Jia and  \v Sver\'ak claim in \cite{JiaSverak,JiaSverak2} that, if $v$ exhibits this decay, then the pressure formula from \cite{KiSe} is valid.  Since the decay property is easier to directly establish for a given solution, this justifies using it in place of the explicit pressure formula in the definition of local Leray solutions.  It turns out that these properties are equivalent when $v_0\in E^2$.  This can be proved using ideas contained in a recent preprint of Maekawa, Miura, and Prange \cite{MMP}  on the construction of local energy solutions in the half space.

Local Leray solutions are known to satisfy a useful a priori bound.  Let $\mathcal N (v_0)$ denote the class of local Leray solutions with initial data $v_0$. The following estimate is well known for local Leray solutions (see \cite{JiaSverak}): for all $\tilde v\in \mathcal N (v_0)$ and $r>0$ we have
\begin{equation}\label{ineq.apriorilocal}
\esssup_{0\leq t \leq \sigma r^2}\sup_{x_0\in \RR^3} \int_{B_r(x_0)}\frac {|\tilde v(x,t)|^2} 2 \,dx+ \sup_{x_0\in \RR^3}	\int_0^{\sigma r^2}\int_{B_r(x_0)} |\nabla \tilde v|^2\,dx\,dt <C {A} ,
\end{equation}
where 
\begin{equation}\label{ineq.apriorilocalconstant}
{
A = \sup_{x_0\in \RR^3} \int_{B_r(x_0)} \frac {|v_0|^2}2\,dx,\quad
 \sigma(r) =c_0\, \min\big\{r^2A^{-2} , 1  \big\},}
\end{equation}
for a small universal positive constant $c_0$.

Concurrently to the publication of \cite{BT1}, Lemari\'e-Rieusset published the book \cite{LR2}, which includes a chapter on the self-similar solutions of \cite{JiaSverak}.  Here, Lemari\'e-Rieusset generalizes the space of initial data to include any $L^2_{\loc}$, divergence free, self-similar vector field. The main elements of his argument are as follows. He first uses the Leray-Schauder approach of \cite{JiaSverak} to construct self-similar solutions for initial data $v_0$ satisfying $|v_0(x)|\lesssim |x|^{-1}$.  This construction is more general than that in \cite{JiaSverak} but less general than that in \cite{BT1}.  But, provided $v_0$ is self-similar, $v_0\in L^2_{\loc}$ if and only if $v_0\in L^2_{\uloc}$.  And, furthermore, if $v_0$ is self-similar and belongs to $L^2_{\uloc}$, then it can be approximated by a sequence $v_0^{(k)}$ where each $|v_0^{(k)}(x)|\lesssim |x|^{-1}$. Then, the first construction gives local Leray solutions for each $v_0^{(k)}$ and, because local Leray solutions satisfy the a priori bound \eqref{ineq.apriorilocal} depending only on the $L^2_{\uloc}$ norm of their initial data, these will converge to a SS local Leray solution with $L^2_{\loc}$ data. This argument breaks down for DSS solutions since $L^2_{\loc}\cap DSS \neq L^2_{\uloc}\cap DSS$ (see \eqref{DSS-example} for an example) and, therefore, we cannot get the uniform bound \eqref{ineq.apriorilocal} on a sequence of approximating solutions for free. 

Chae and Wolf, on the other hand, introduced an entirely new method in \cite{Chae-Wolf} to construct $\la$-DSS solutions for any $\la>1$ and initial data $v_0\in L^2_{\loc}(\R^3)$. These solutions live in the class of ``local Leray solutions with projected pressure,'' which means they satisfy a modified local energy inequality instead of the classical local energy inequality \eqref{CKN-LEI} of \cite{CKN}. To construct these solutions, Chae and Wolf use a fixed point argument to solve the mollified Navier-Stokes equations (this is the same system studied in \cite{BT1}, but written in physical variables as opposed to the similarity variables, {see \eqref{mollified1} and \eqref{mollified2}}).   To apply the fixed point argument, Chae and Wolf first prove existence for the (mollified) linearized equations where the given drift velocity is DSS.  They then apply a fixed point theorem (the space for the fixed point argument is a bounded set of the DSS subspace of $L^{18/5}(0,T;L^3(B_1))$ -- $B_r$ denotes the ball of radius $r$ centered at the origin -- defined below \cite[(3.1)]{Chae-Wolf}) to prove that there exists a drift velocity which matches the solution.  This gives existence of a DSS solution to the mollified Navier-Stokes equations.  Note that the approximations satisfy the a priori (energy) bound \cite[(2.35)]{Chae-Wolf} and the norm of the mollification term can be absorbed for $T$ sufficiently small.

In this paper we give a simple, alternative proof of the result in \cite{Chae-Wolf}. The following theorem is our main result.

\begin{theorem}\label{thrm.main}
Assume $v_0\in L^2_{\loc}(\R^3)$ is a divergence free $\la$-DSS vector field for some $\la>1$.  Then there exists a $\la$-DSS distributional solution $v$ to \eqref{eq.NSE} and associated pressure $\pi$ so that $v$ is suitable in the sense of \cite{CKN} and satisfies
\[
\lim_{t\to 0^+}\|v(t)-v_0\|_{L^2(K)} = 0,
\]
for every compact subset $K$ of $\R^3$. Moreover, for any $T>0$ and compact subset $K$ of $\R^3$, we have $v\in L^\I(0,T;L^2(K))\cap L^2(0,T;H^1(K))$ and $\pi \in L^{3/2}(0,T;L^{3/2}(K))$.   Furthermore, for any $(x,t)\in \R^3\times (0,\I)$, the pressure satisfies the following formula
\EQ{ \label{eq.pressure0}
	\pi(x,t) =&  -\frac 1 3 {|v|^2(x,t)} 
	\\
	&+ \lim_{\delta \to 0} \int_{|y|>\delta} K_{ij}(x-y)v_i(y,t) v_j(y,t)\,dy,
}
in  $L^{3/2}_{\loc}(\R^3\times (0,\I))$.
\end{theorem}

 \noindent \emph{Comments on Theorem \ref{thrm.main}}
 \begin{enumerate}
 
 \item In \cite{Chae-Wolf}, the data also belongs to $L^2_{\loc}$, but the solution is not shown to satisfy the local energy inequality of \cite{CKN}. Instead, it satisfies a ``local energy inequality with projected pressure''.  Since the solution constructed in Theorem \ref{thrm.main} satisfies the traditional local energy inequality, this theorem is a slight refinement of the main result of \cite{Chae-Wolf}. Furthermore, we are careful to give a precise {formulation \eqref{eq.pressure0}} of the pressure and its connection to the velocity.  The relationship between $v$ and $\pi$ is less clear in \cite{Chae-Wolf}.

\item {The integral in  \eqref{eq.pressure0} 
is not a Calderon-Zygmund singular integral because we do not have a global bound of $v$. It is defined in $L^{3/2}_{\loc}$ using the DSS property.}

 \item Our method of proof is by approximation and is similar to the argument from \cite{LR2}.  The main difference is that we need to construct a sequence of approximating solutions and establish a new \emph{a priori} bound for these solutions for DSS data -- in \cite{LR2} the bound \eqref{ineq.apriorilocal} is sufficient (and free).  Note that an approximation argument using \eqref{ineq.apriorilocal} was also used by the authors in \cite{BT1} to construct SS solutions as a limit of DSS solutions where the scaling factors are converging to $1$.
 
 \item {Generally, the solution $v$ is not necessarily a local Leray solution because $v_0$ may not be in $L^2_\uloc$, and we do not assert the uniform bounds {(item 2)} in Definition \ref{def:localLeray}.   
Consider the  DSS function in $L^2_{\loc}$ for $0<a<\frac 32$,
\begin{equation}\label{DSS-example}
f_a(x) = \sum_{k\in \mathbb Z}\la^k f_{a,0}(\la^k x), \quad
f_{a,0}(x)  = |x-x_0|^{-a} \chi(x-x_0), 
\end{equation}
where $1+r<|x_0|<\la-r$ for some $r>0$,  and $\chi$ is the characteristic function of the ball $B_r(0)$.
It is not in $L^2_{\uloc}$ when $1<a<\frac 32$ for its behavior at infinity. 
It is in $L^2_{\uloc}$ when $0<a\le 1$.
The function $f_1(x)$ for $a=1$ is given in Comment 4 after \cite[Theorem 1.2]{BT2} as an inapplicable example since it is not in $L^{3,\infty}(\R^3)$. 

\item  If $v_0\in L^2_{\uloc}$, then it is not difficult to obtain uniform bounds on $v$ in the sense of item 2 from Definition \ref{def:localLeray}. Furthermore,  item 5 from Definition \ref{def:localLeray} can be established whenever {$v_0\in E^2$} (see \cite{MMP}).   
Thus, our construction yields DSS local Leray solutions whenever the data is DSS, divergence free, and in $E^2$.  
 }
 \end{enumerate}

Our strategy for proving Theorem \ref{thrm.main} is to approximate a solution with data in $L^2_{\loc}$ using solutions constructed in \cite{BT1}.  There are several steps.  First we need to prove that DSS data in $L^2_{\loc}$ can be approximated in $L^2(B_1)$ by DSS data in $L^3_w$.  This is the subject of \S\ref{sec.approx}.  Then, \cite{BT1} gives us a sequence of DSS solutions in the local Leray class.  To prove that these solutions converge to a solution with $L^2_{\loc}$ data satisfying the desired pressure formula, we need to establish new a priori bounds for the solutions from \cite{BT1} which are independent of the $L^3_w$ norm of the initial data (this is done in \S\ref{sec.bound}) and also prove that they satisfy the pressure formula (see \S\ref{sec.pressure}).   In \S\ref{sec.extension} and \S\ref{sec.construction}, we put these ingredients together to prove Theorem \ref{thrm.main}.

 \section{A limiting pressure formula for DSS solutions}\label{sec.pressure}
 
 In this section we will prove that, under certain conditions, the limiting pressure distribution of an approximation scheme for \eqref{eq.NSE} inherits the structure of the approximate pressure distributions.  This result will be applied in \S\ref{sec.bound} and \S\ref{sec.construction}.
 
 \begin{lemma}\label{lemma.pressure}Fix $\la>1$ and $T>0$.
 	Let $v_0\in L^2_{\loc}$ be a given divergence free, $\la$-DSS vector field and assume $\{v_0^{(k)}\} \subset L^2_{\loc}$ is a sequence of divergence free, $\la$-DSS vector fields so that $v_0^{(k)}\to v_0$ in $L^2(B_1)$.
 	Assume $v_k$ and $\tilde v_k$ are divergence free, $\la$-DSS vector fields and that there exists a distribution $\pi_k$ so that the following conditions are satisfied:
 	\begin{itemize}
 		\item $v_k$, $\tilde v_k$, and $\pi_k$ solve the system
 		\[
 		\partial_t v_k -\Delta v_k +\tilde v_k\cdot\nabla v_k+\nabla \pi_k=0\qquad (x,t)\in \R^3\times [0,T],
 		\]
 		for the initial data $v_0^{(k)}$ and both $v_k$ and $\tilde v_k$ converge to $v_0^{(k)}$ in $L^2_{\loc}$.
 		\item  $v_k$ and $\tilde v_k$ are uniformly bounded in ${L^\I(0,T;L^2(B_1))}\cap  {L^2(0,T;H^1(B_1))}$ over all $k\in \N$.
 		\item for all $0<t\leq T$, $\pi_k$ satisfies the formula
 		\EQ{\label{eq.pressure2}
 			\pi_k(x,t) =&  -\frac 1 3 {[\tilde v_k \cdot v_k](x,t)} 
 			\\
 			&+ \lim_{\delta \to 0} \int_{|y|>\delta} K_{ij}(x-y)(\tilde v_k)_i(y,t) (v_k)_j(y,t)\,dy.
 		}
 		\item  there exists a $\la$-DSS solution $v$ in ${L^\I(0,T;L^2(B_1))}\cap  {L^2(0,T;H^1(B_1))}$ with pressure $\pi$ in $L^{3/2}(0,T;L^{3/2})$ so that
 		\begin{align*}
 		&v_k \text{ and } \tilde v_k \to v \mbox{ weakly in } L^2(0,T;H^1(B_1))
 		\\& v_k \text{ and } \tilde v_k \to v \mbox{ in } L^2(0,T;L^2(B_1)) 
 		\\& \pi_k \to \pi \text{ weakly in } L^{3/2}(0,T;L^{3/2}(B_1)).
 		\end{align*}
 	\end{itemize} 
 	Then, for a.e.~$0<t\leq T$ and $x\in B_{\la}$, the pressure 
 	$\pi$ satisfies the formula
 	\EQ{\label{eq.pressure3}
 		\pi(x,t) =&  -\frac 1 3 {|v|^2(x,t)} 
 		\\
 		&+ \lim_{\delta \to 0} \int_{|y|>\delta} K_{ij}(x-y)( v)_i(y,t) (v)_j(y,t)\,dy,
 	}
	in $L^{3/2}((0,T)\times B_\la)$.  
 \end{lemma}

\begin{remark}
The purpose of this lemma is to establish the pressure formula \eqref{eq.pressure3} which, ultimately, will allow us to prove \eqref{eq.pressure0}.  It is, however, not needed to establish the other conclusions of Theorem \ref{thrm.main}.
\end{remark}

 \begin{proof}
 	Note that since $v_k$, $\tilde v_k$, and $v$ are all uniformly bounded in ${L^\I(0,T;L^2(B_1))}\cap  {L^2(0,T;H^1(B_1))}$,   convergence in $L^2(0,T;L^2(B_1))$, H\"older's inequality, Sobolev embedding, using the equation to get uniform bound of $\pd_t v_k$, and re-scaling the solution, implies that 
 	\[
 	v_k \text{ and } \tilde v_k \to v \mbox{ in } L^3(0,T;L^3(B_1)).
 	\] 
	It also shows that $v_k$, $\tilde v_k$, and $v$ are all uniformly bounded in $L^3(0,T;L^3(B_1))$ (at least for $k$ sufficiently large).
 	
 	Let
 	\[ 
 	\pi_k^1(x,t) =  -\frac 1 3 {[\tilde v_k \cdot v_k](x,t)}, 
 	\]
 	\[
 	\pi_k^2(x,t) =  \lim_{\delta \to 0} \int_{\la^2>|y|>\delta} K_{ij}(x-y)(\tilde v_k)_i(y,t) (v_k)_j(y,t)\,dy,
 	\]
 	and 
 	\[
 	\pi_k^3(x,t)=\int_{y\geq \la^2} K_{ij}(x-y)(\tilde v_k)_i(y,t) (v_k)_j(y,t)\,dy.
 	\]
 	Also let 
 	\[ 
 	\pi^1(x,t) =  -\frac 1 3 {|v|^2(x,t)}, 
 	\]
 	\[
 	\pi^2(x,t) =  \lim_{\delta \to 0} \int_{\la^2>|y|>\delta} K_{ij}(x-y)v_i(y,t) v_j(y,t)\,dy,
 	\]
 	and 
 	\[
 	\pi^3(x,t)=\int_{y\geq \la^2} K_{ij}(x-y)v_i(y,t) v_j(y,t)\,dy.
 	\]
 	Since $v_k \text{ and } \tilde v_k \to v \mbox{ in } L^3(0,T;L^3(B_\la))$, we have $\pi_k^1\to \pi^1$ in $L^{3/2}(0,T;L^{3/2}(B_\la))$.  
	
	Let 
 	\[
{ 	h_{i,j}(y,t)= (\tilde v_k)_i ( v_k)_j-v_i v_j  =  \bket{ (\tilde v_k)_i [(v_k)_j-v_j  ]  + [(\tilde v_k)_i -v_i  ] v_j } (y,t).}
 	\]
	Using the Calderon-Zygmund theory we clearly have
 	\begin{align}
 		&\int_0^T \int_{B_\la} |\pi_k^2(x,t) - \pi^2(x,t)|^{3/2}\,dx\,dt \nonumber
 		\\&\leq C \int_0^T\int_{B_{\la^2}} |h_{i,j}(x,t)|^{3/2}\,dx\,dt \nonumber
 		\\&\leq C \bigg(  \int_0^T\int_{B_{\la^2}}		\tilde v_k^3		\,dx\,dt\bigg)^{1/2}\bigg( \int_0^T\int_{B_{\la^2}}   (v_k-v)^3	\,dx\,dt\bigg)^{1/2} \nonumber
 		\\ &\quad + C\bigg(  \int_0^T\int_{B_{\la^2}}		v^3		\,dx\,dt\bigg)^{1/2}\bigg( \int_0^T\int_{B_{\la^2}} (\tilde v_k - v )^3   	\,dx\,dt\bigg)^{1/2}.
		\label{pidiff.est}
 	\end{align}
 	Re-scaling gives
 	\[
 				\int_0^T\int_{B_{\la^2}} (\tilde v_k - v )^3(x,t)   	\,dx\,dt = \la^4 \int_0^{T\la^{-4}}\int_{B_1} (\tilde v_k - v )^3(z,\tau)   	\,dz\,d\tau,
 	\]
 	for the obvious choice of $z$ and $\tau$. Since the right hand side of the above equation vanishes as $k\to \I$, as does the identical term but with $\tilde v_k$ replaced by $v_k$, we conclude that $\pi_k^2$ converges to $\pi^2$ in  $L^{3/2}(0,T;L^{3/2}(B_1))$.

 	Establishing the convergence of $\pi_k^3$ to $\pi^3$ is more difficult. Let
 	\[
 	p_k(x,t)=\pi^3_k(x,t)-\pi^3(x,t)
 	=     \int_{|y|\geq \la^2} K_{ij}(x-y) h_{i,j} (y,t) \,dy.
 	\]
 	Fix $x\in B_\la$.  Then
 	\begin{align*}
 		|p_k(x,t)|^{3/2}&\leq C \bigg| \int_{|y|\geq \la^2} \frac 1 {{|y|^3}}	|h_{i,j}(y,t)| \,dy	\bigg|^{3/2}
 		\\&\leq C \bigg(  \int_{|y|\geq \la^2}    \frac 1 {|y|^4}\,dy \bigg)^{1/2}     \int_{|y|\geq \la^2}  \frac 1 {|y|^{5/2}} 	|h_{i,j}(y,t)|^{3/2} \,dy 
 		\\ &= C   \int_{|y|\geq \la^2}  \frac 1 {|y|^{5/2}} 	|h_{i,j}(y,t)|^{3/2} \,dy.
 	\end{align*}
 	Let $A_k = \{x:\la^{k-1}\leq |x|<\la^k\}$ {for $k\in \Z$.}  Then, using the scaling properties of $h$, 
 	\begin{align*}
 	 \int_{|y|\geq \la^2}  \frac 1 {|y|^{5/2}} 	|h_{i,j}(y,t)|^{3/2} \,dy&=  \sum_{k=3}^\I \int_{A_k} \frac 1 {|y|^{5/2}} |h_{i,j}(y,t)|^{3/2} \,dy
 	 \\&\leq C(\la) \sum_{k=3}^\I \frac 1 {\la^{5k/2}}  \int_{A_k} |h_{i,j}(y,t)|^{3/2}\,dy
 	 \\&\leq C(\la) \sum_{k=3}^\I \frac 1 {\la^{5k/2}}  \int_{B_1} |h_{i,j}(z,t\la^{-2k})|^{3/2}\,dz.
 	\end{align*}
	Thus,
	\begin{align*}
	\int_0^T \int_{B_{\la}} |p_k(x,t)|^{3/2}\,dt
	&\leq \la^3 C(\la) \int_0^T \sum_{k=3}^\I \frac 1 {\la^{5k/2}}  \int_{B_1} |h_{i,j}(z,t\la^{-2k})|^{3/2}\,dz \,dt
	\\& \leq  C(\la) {\sum_{k=3}^\I \frac 1 {\la^{ k/2}}\int_0^{T\la^{-2k}}}\! \int_{B_1} |h_{i,j}(z,\tau)|^{3/2}\,dz \,d\tau
	\\&\leq  C(\la) \int_0^{T} \int_{B_1} |h_{i,j}(z,\tau)|^{3/2}\,dz \,d\tau.
	\end{align*}
{Thus this term is bounded as \eqref{pidiff.est}.  	
}

We have now shown that
 	$\pi_k(x,t)  $ converges weakly to both  $\pi^1(x,t)+\pi^2(x,t)+\pi^3(x,t)$ and $\pi(x,t)$ in $L^{3/2}(0,T;L^{3/2}(B_\la))$, implying that $\pi(x,t)=\pi^1(x,t)+\pi^2(x,t)+\pi^3(x,t)$ as distributions.   In other words,
{$\pi(x,t)$ satisfies \eqref{eq.pressure3}}
in
$L^{3/2}((0,T)\times B_\la)$.   
 \end{proof}

\section{Properties of DSS solutions with data in $L^3_w$}\label{sec.bound}

The goal of this section is to obtain a bound on the local evolution of DSS solutions $v$ constructed in \cite{BT1} that is independent of both the $L^3_w$ and $L^2_{\uloc}$ norms of $v$ and to establish an explicit representation formula for the pressure.  

 %
Assume $v_0\in L^3_w(\R^3)$ and $v$ is {a} DSS solution evolving from $v_0$ as constructed in \cite{BT1}.  For a generic solution to \eqref{eq.NSE}, we cannot close energy estimates for $\phi v$ solely in terms of $v_0|_{B_\la}$ -- there is always some spillover.  
Proposition \ref{lemma.energy} states that this is possible for DSS solutions as a result of their scaling properties.   In our argument, we must work with a quantity that is continuous in time.  This is not known for $\int_{B_1} |v(t)|^2\,dx$ when $v$ is a local Leray solution.  Hence, we need to work at the level of a \emph{mollified approximation scheme} \cite[(2.24)]{BT1},  see \eqref{mollified1} below.  Note that in \cite{BT1}, the mollified scheme is used to approximate a solution to the time-periodic Leray equations and the mollification is time-independent.  Undoing the similarity transformation results in a time-dependent mollification of the drift component of the nonlinear term of the solution in the physical variables,  see \eqref{mollified2} below; this matches the mollification used in \cite{Chae-Wolf}.

%

\begin{proposition}\label{lemma.energy}
Fix $\la>1$. 
Assume $v_0\in L^3_w(\R^3)$ is  $\la$-DSS and divergence-free,
and $v$ is {a} $\la$-DSS {local Leray} solution evolving from $v_0$ constructed in \cite{BT1} (in particular, it is the limit of the mollified approximation scheme \cite[(2.24)]{BT1}) and $\pi$ is its associated pressure. 
Let $\al_0=\| v_0\|^2_{L^2(B_\la)}$.
 Then, there exist positive $T=T(\al_0,\la)$ and $C(\al_0,\la)$ independent of $\norm{v_0}_{L^2_\uloc}$ and $\norm{v_0}_{L^3_w}$
 so that
\EQ{ \label{prop3.1-1}
	\esssup_{0\leq t\leq T} \int_{B_1} |v(x,t)|^2\,dx + \int_0^T\int_{B_1} |\nabla v|^2\,dx\,dt <C(\al_0,\la),
}
and 
\EQ{ \label{prop3.1-2}
\int_0^T \int_{B_1}| \pi(x,t)|^{3/2}\,dx\,dt  <C(\al_0,\la).
}
Moreover, for $x\in B_1$ and $t\in (0,T)$, the pressure satisfies the formula  
\EQ{ \label{prop3.1-3}
	\pi(x,t) =&  -\frac 1 3 {|v|^2(x,t)} 
	\\
	&+ \lim_{\delta \to 0} \int_{|y|>\delta} K_{ij}(x-y)v_i(y,t) v_j(y,t)\,dy,
}
in $L^{3/2} (B_1 \times (0,T))$.
\end{proposition}

{Typically, the best pressure decompositions we have for local Leray solutions depend on a particular ball containing the spatial point at which the pressure is being computed.  The resulting formula consists of a local Calderon-Zygmund part and a far-field part with a singular kernel that is decaying faster than $K$ kernel.  The formula \eqref{prop3.1-3} does not involve such a decomposition, and, as is evident in the proof, the integral in \eqref{prop3.1-3} is defined using the DSS property.  
}

The proof of \cite{BT1} shows that the left sides of \eqref{prop3.1-1} and \eqref{prop3.1-2} are bounded by constants depending on $v_0$, in particular its $L^3_w(\R^3)$-norm. For this application, we need a bound depending only on $\norm{v_0}_{L^{2}(B_\la)}$ and $\la$.

\begin{proof} 
Since $v$ is a solution from \cite{BT1}, its image under the similarity transform \eqref{variables} solves the time-periodic Leray equations and is the limit of a mollified approximation scheme \cite[(2.24)]{BT1}.  In particular, for each $\e>0$, there exists a time periodic solution $u_\e$ to the problem
\begin{equation}\label{mollified1}
	\bke{\partial_s u_\e  -\Delta u_\e -\frac 1 2 u_\e -\frac 1 2 y\cdot\nabla u_\e  + (\eta_\e * u_\e)\cdot\nabla u_\e  +\nabla p_\e} (y,s)= 0,
\end{equation}
where $\eta_\e (y)=\frac 1 {\e^3}\eta(y/\e)$ and $\eta\in C_0^\I(\R^3)$, is non-negative, and satisfies $\int \eta(y) \,dy =1$.  Applying \eqref{ansatz2}-\eqref{variables} we obtain a $\la$-DSS vector field $v_\e$ satisfying
\begin{equation}\label{mollified2}
	\partial_t v_\e(x,t) -\Delta v_\e(x,t)+(\eta_{\e\sqrt{t}} * v_\e) \cdot \nabla v_\e (x,t)+\nabla \pi_\e (x,t)= 0.
\end{equation}
Note the time dependence of the convolution kernel $\eta_{\e\sqrt{t}} $ in \eqref{mollified2}.

By the convergence properties of $u_\e(y,s)$ to $u(y,s)=\sqrt{t} v (x,t)$ \cite[p.\,1108]{BT1} and discretely self-similar scaling (to extend the estimates down to $t=0$), it follows that for all $T>0$ {and all compact sets  $K\subset \R^3$,}
\begin{align*}
&v_\e \to v \mbox{ weakly in } L^2(0,T;H^1 {(K)}),
\\& v_\e \to v \mbox{ strongly in } L^2(0,T;L^2(K)) ,
\\& v_\e(s)\to v(s) \mbox{ weakly in } L^2 {(K)}  \mbox{ for all }s\in [0,T].
\end{align*}
Note also that $v_\e(t)\to v_0$ in $L^2_{\loc}$, i.e.\,the mollification does not {affect} the initial data.
Furthermore, {because} each $v_\e$ is smooth on $\R^3\times (0,\I)$ and right continuous in $L^2_{\loc}$ at $t=0$,  it follows that 
\[\al_\e(t)= \int_{B_1}|v_\e(x,t)|^2\,dx,\] 
and
\[
\tilde \al_\e (t)=\sup_{0\leq \tau\leq t}\al_\e(\tau)
\]
are continuous as functions of $t$. This is not clearly true for $\int_{B_1} |v(x,t)|^2\,dx$. 

 {Note that, for any $k \in \Z$ and $q\in [1,\infty)$, since $v_\e(x,t)  = \la^{-k} v_\e(\la^{-k}x , \la^{-2k} t)$,
\EQ{\label{eq3.6}
 \int_{B_{\la^k} }|v_\e(x,t)|^q\,dx = \la^{(3-q)k} \int_{B_{1}} |v_\e(\td x,\la^{-2k} t)|^q\,d\td x .
}
}

Our goal is to establish local in time a priori bounds for $\al_\e(t)$ that are independent of $\e$.
Note that $v_\e$ satisfies the local energy equality, i.e.,
\begin{align}
\notag &\int |v_\e|^2\phi(t) \,dx +2\int_0^t \int |\nabla v_\e|^2\phi\,dx\,ds
\\ \label{ineq.local2}&=
\int |v_0|^2\phi \,dx+ \int_0^t\int |v_\e|^2 (\partial_s\phi +\Delta\phi) \,dx\,ds 
 \\ \notag&+\int_0^t\int (|v_\e|^2 ((\eta_{{\e\sqrt{s}}} * v_\e) \cdot \nabla \phi)\,dx\,ds
+\int_0^t\int2\pi_\e(v_\e\cdot \nabla\phi)\,dx\,ds,
\end{align}
for any non-negative $\phi\in C^\I_0(\R^3\times [0,\I))$. 
 {Fix $\chi \in C^\infty(\R)$ with $\chi(t)=1$ if $t \le 1$ and $\chi(t)=0$ if $t\ge \la$. We now fix $\phi$ in \eqref{ineq.local2} as
\[
\phi(x,t) = \chi^2(|x|) \cdot \chi(t).
\] 
 
We will estimate the terms on the right hand side of \eqref{ineq.local2}  for $0< t\le 1$, and we can treat $\phi$ as $t$-independent from now on.}
The first term is bounded by $\al_0$.
For the second, using the scaling properties \eqref{eq3.6} of $v_\e$, we have
\begin{align*}
\int_0^t\int |v_\e|^2(\partial_s\phi + \Delta \phi )\,dx\,ds&\leq C  \int_0^t \int_{B_\la} |v_\e|^2\,dx\,ds
\\&\leq  C  {\la^3} \int_0^{t/\la^2}\int_{B_1} |v_\e|^2\,dx\,ds
\\&\leq C( \la) \int_{0}^t \tilde \al_\e(s)\,ds.
\end{align*}
For the cubic term, we  begin by using Young's inequality to obtain
\begin{align*} 
\int_0^t\int |v_\e|^2((\eta_{\e\sqrt{s}}*v_\e)\cdot \nabla\phi)\,dx\,ds
&\leq C  \int_0^t\int_{B_\la} |v_\e|^3 \,dx\,ds 
\\&+ C  \int_0^t\int_{B_\la}   |(\eta_{\e\sqrt{s}}*v_\e) |^3\,dx\,ds.
\end{align*}
Re-scaling the non-mollified term and making the obvious change  of variables results in the estimate
\begin{align*}
\int_0^t\int_{B_\la} |v_\e|^3 \,dx\,ds \leq  C( \la)\int_0^{t/\la^2}\int_{B_1} |v_\e|^3\,dy\,d\tau \leq C ( \la)\int_0^t\int |v_\e|^3\phi^{3/2} \,dx\,ds.
\end{align*}
For the term involving the mollifier, note that $\eta\in C_0^\I$ and $\supp \eta \subset B_\rho$ for some $\rho>0$.  By taking $\e$ sufficiently small we can ensure that, whenever $s<1$, $\supp \eta_{\e\sqrt{s}}\subset  {B_{\la-1}}$.   Note $\la^k + (\la-1) \le \la^{k+1}$ for all $k \ge 0$.
Thus, for $x \in B_\la$, 
\begin{align*}
|	 (\eta_{\e\sqrt{s}} * v_\e ) (x,s) |&\leq \int \eta_{\e\sqrt{s}}(y) |v_\e(x-y,s) |\,dy
\\&=\int \eta_{\e\sqrt{s}}(y) |v_\e(x-y,s) | \chi_{B_{\la^2}} (x-y)\,dy
\\&=(\eta_{\e\sqrt{s}} *(\chi_{B_{\la^2}}| v_\e |))(x,s),
\end{align*}
whenever $\e$ is sufficiently small and $s< 1$. Therefore, under the same assumptions and after re-scaling we see that, for any $1<q<\I$,
\begin{align}\label{ineq.mollifier}
\|	 (\eta_{\e\sqrt{s}} * v_\e ) (s) \|_{L^q(B_{\la})}\leq C( {q,}\eta) \|v_\e(s)\|_{L^q(B_{\la^2})} \leq C({q,}\eta,\la) \|v_\e {(\la^{-4}s)}\|_{L^q(B_1)},
\end{align} 
where $C$ is independent of $s$ and $\e$.   Note that this estimate is also valid if $B_\la$ is replaced by $B_{\la^2}$ but with a different choice of constants, smallness condition on $\e$, and right hand side determined at time $\la^{-6}s$.    

Using standard inequalities and \eqref{ineq.mollifier} with $q=3$  thus leads to the estimate
\begin{align}
\label{ineq.cubic}\int_0^t\int |v_\e|^2((\eta_{\e\sqrt{s}}*v_\e)\cdot \nabla\phi)\,dx\,ds \leq C (\eta,\la) \int_0^t\int |v_\e|^3\phi^{3/2} \,dx\,ds.
\end{align}
By the Gagliardo-Nirenberg inequality and re-scaling \eqref{eq3.6}, we have, for any $s>0$, that 
\EQN{
\|\phi^{1/2} v_\e(s)\|_{L^3}&\leq C \|\nabla \otimes (\phi^{1/2} v_\e)\|_{L^2}^{1/2} \|\phi^{1/2} v_\e\|_{L^2}^{1/2}(s)
\\& \leq  C( \la)\big(  \tilde \al_\e(s)^{1/2} + \|\phi^{1/2} \nabla v_\e (s) \|_{L^2} \big)^{1/2}\big(  \tilde \al_\e(s)  \big)^{1/4}.
} 
Hence, for any $\ga>0$, 
\[
\|\phi^{1/2} v_\e(s)\|_{L^3}^3\leq C(\la) \big(\ga^{-3}\tilde \al_\e(s)^{3} +  \ga \tilde \al_\e(s)^{{1}} +\ga \|\phi^{1/2} \nabla v_\e (s) \|_2^2    \big).
\]
Thus,
\EQ{
\label{ve3est}
 \int_0^t\int |v_\e|^2((\eta_{\e\sqrt{s}}*v_\e)\cdot \nabla\phi)\,dx\,ds
 &\leq C( \la,\ga,\eta) \int_0^t \big( \tilde \al_\e(s)^3+\tilde \al_\e(s)^1 \big)\,ds 
 \\&+ C( \la) \ga\int_0^t\int |\nabla v_\e|^2\phi\,dx\,ds.
}
Provided $\ga$ is small enough, the gradient term can be absorbed into the left hand side of \eqref{ineq.local2}.

We next estimate the pressure term.  
For this we need a formula for the pressure which we presently justify.
Let $w_\e=v_\e-V_0$ where $V_0(x,t)=e^{t\Delta}v_0$. We have
\[
\partial_t{w_\e}-\Delta w_\e +\nabla \pi_\e  =g, \quad \div w_\e=0,
\]
where $g_i = -\pd_j G_{ji}$ with
\EQN{
G &=(\eta_{\e\sqrt t} *v_\e )\otimes v_\e
\\
& =  (\eta_{\e\sqrt t} *w_\e + \eta_{\e\sqrt t} *V_0 ) \otimes  (w_\e  +V_0).
}
 {For $0<t_1<t_2<\I$, we have 
\EQN{
V_0 &\in C([t_1,t_2];L^4(\R^3)\cap L^\I(\R^3)), \\
w_\e &\in L^\infty(t_1,t_2;L^2 (\R^3))\cap L^2 (t_1,t_2;L^6 (\R^3)) \subset L^4(t_1,t_2;L^3(\R^3)).
}
%
By Young's convolution inequality, 
\[
\norm{G}_{L^2(t_1,t_2;L^2)} \lec \norm{\eta_{\e\sqrt t}}_{L^\I(t_1,t_2; L^{6/5} \cap L^1)}
\bke{ \norm{w_\e}_{L^4(t_1,t_2;L^3(\R^3))} + \norm{V_0}_{L^4(\R^3\times [t_1,t_2])}}^2.
\]
}%
Since $g\in L^2([t_1,t_2];H^{-1})$,  \cite[Lemma A.2]{CKN} implies $w_\e  \in C([t_1,t_2];L^2)$ (after modification on a set of time of measure zero; since the modified vector field still satisfies the above system distributionally, this does not effect our argument).

Consider the following non-stationary Stokes system with forcing $g$
\[
\partial_t V-\Delta V  +\nabla P = g, \quad \div V=0,
\] 
with initial data $V_0=w_\e(t_1)\in L^2(\R^3)$. 
It is well known that if $g\in L^\I(t_1,t_2;H^{-1})$ and $V_0\in L^2$, then there exists a unique $V\in C_w([t_1,t_2];L^2(\R^3))\cap L^2([t_1,t_2];H^1(\R^3))$ and unique $\nabla P$ solving the above non-stationary stokes system (see \cite[p. 1107-1108]{BT1}). Letting  $V=w_\e$ and $P=\pi_\e$, this implies that $w_\e$ and $\nabla \pi_\e$ are unique. Up to a function $\pi_*(t)$ independent of $x$,
%
%
 \EQ{\label{eq.pressure}
 	\pi_\e(x,t) -\pi_*(t)=&  -\frac 1 3 {[(\eta_{\e\sqrt{t}}*v_\e)\cdot v_\e](x,t)} 
 	\\
 	&+ \lim_{\delta \to 0} \int_{|y|>\delta} K_{ij}(x-y)(\eta_{\e\sqrt{t}}*v_\e)_i(y,t) (v_\e)_j(y,t)\,dy,
 }
where $K_{ij}(x)=\partial_i\partial_j \frac 1 {4\pi |x|}$. The right side is defined in $L^2([t_1,t_2];L^2(\R^3))$.
Since the only appearance of $\pi_\e$ in \eqref{mollified2} is $\nabla \pi_\e$, we can re-define $\pi_\e$ to equal $\pi_\e-\pi_*(t)$ and, therefore, can drop $\pi_*(t)$ from \eqref{eq.pressure}.

 {The pressure $\pi_\e$ given by \eqref{eq.pressure} is already bounded in $L^2([t_1,t_2];L^2(\R^3))$ for any $0<t_1<t_2<\I$ but the bound depends on $t_1$, $t_2$ and $\e$. We now bound it in $L^{3/2}(0,T; L^{3/2}(B_\la))$.}
Bounding the first term from \eqref{eq.pressure} is simple given H\"older's inequality, \eqref{ineq.mollifier}, and \eqref{ineq.cubic}.  In particular, we have for any $\ga>0$
\begin{align*}
&\int_0^t	\|\frac 1 3 |(\eta_{\e\sqrt{s}}*v_\e)(\cdot,s)||v_\e(\cdot,s)| \|_{L^{3/2}(B_\la)}^{3/2} \,ds 
\\ &\qquad
\leq   C( \la,\ga,\eta) \int_0^t \big( \tilde \al_\e(s)^3+\tilde \al_\e(s)^1 \big)\,ds
 + \ga\int_0^t\int |\nabla v_\e|^2\phi\,dx\,ds.
\end{align*}
To bound the principal value integral in \eqref{eq.pressure}, we need to split the integral into  local and non-local parts as follows,
\begin{align*}
&\lim_{\delta \to 0} \int_{|y|>\de} K(x-y)(\eta_{\e\sqrt{t}}*v_\e)(y,t) v_\e(y,t)\,dy 
\\&=
\lim_{\delta \to 0} \int_{ B_{\la^2}\setminus {B_\delta}} K(x-y)(\eta_{\e\sqrt{t}}*v_\e)(y,t) v_\e(y,t) \chi_{B_{\la^2}}(y)\,dy 
\\&+    \int_{|y|>\la^2} K(x-y)(\eta_{\e\sqrt{t}}*v_\e)(y,t) v_\e(y,t)\,dy  
\\&=: \pi_{\mathrm{near}}(x,t) + \pi_{\mathrm{far}}(x,t).
\end{align*}
To bound $\pi_{\mathrm{near}}$ note that, by the Calderon-Zygmund theory,
\begin{align*}
\|\pi_{\mathrm{near}} (\cdot,t)\|_{L^{3/2}(B_\la)} &\leq   \| (\eta_{\e\sqrt{t}}*v_\e)(\cdot,t)  v_\e(\cdot,t)\|_{L^{3/2}(B_{\la^2})},
\end{align*}
and, arguing as above using \eqref{ineq.mollifier} but with $B_{\la^2}$ in place of $B_\la$ (see the note following \eqref{ineq.mollifier}), it follows that
\[
	\int_0^t\|\pi_{\mathrm{near}} (\cdot,s)\|_{L^{3/2}(B_\la)}^{3/2}	 \,ds \leq  C(\la,\ga,\eta) \int_0^t \big( \tilde \al_\e(s)^3+\tilde \al_\e(s)^1 \big)\,ds+ \ga\int_0^t\int |\nabla v_\e|^2\phi\,dx\,ds.
\]

Bounding the term $\pi_{\mathrm{far}}$ is more complicated.  Let $A_k=\{x: \la^{k-1}\leq |x|< \la^k  \}$. We start with the following pointwise estimate which is valid whenever $x\in B_{\la}$,
\begin{align*}
|\pi_{\mathrm{far}} (x,t)| &\leq C\sum_{k=3}^\I \int_{A_k} \frac 1 {|x-y|^3} |(\eta_{\e\sqrt{t}}*v_\e)(y,t)|\,|  v_\e(y,t)|\,dy
\\&\leq C(\la) \sum_{k=3}^\I \frac 1 {\la^{3k}}  \int_{A_k} |(\eta_{\e\sqrt{t}}*v_\e)(y,t)|\,|  v_\e(y,t)|\,dy
\\&=C(\la) \sum_{k=3}^\I \frac 1 {\la^{2k}}  \int_{A_0} |(\eta_{\e\sqrt{t\la^{-2k}}}*v_\e)(z,t\la^{-2k})|\,|  v_\e(z,t\la^{-2k})|\,dz
\\&\leq C(\la)\sum_{k=3}^\I \frac 1 {\la^{2k}}   \|(\eta_{\e\sqrt{t\la^{-2k}}}*v_\e)(t\la^{-2k})   \|_{L^2(B_1)}   \|  v_\e (t\la^{-2k})\|_{L^2(B_1)}
\\&\leq C(\la)\sum_{k=3}^\I \frac 1 {\la^{2k}}   \|  v_\e (t\la^{-2k})\|_{L^2(B_{\la^2})}^2
\\&\leq C(\la) \tilde \al_\e(t),
\end{align*}
where we have used \eqref{eq3.6}, \eqref{ineq.mollifier}  and re-scaled the solution. 
Therefore,
\[
	\int_0^t\|\pi_{\mathrm{far}} (\cdot,s)\|_{L^{3/2}(B_\la)}^{3/2}	 \,ds \leq  C(\la) \int_0^t  \tilde \al_\e(s)^{3/2}\,ds.
\]
After using H\"older's inequality, \eqref{ineq.cubic}, the above bounds, and  {$\al^{3/2} \le \al+\al^3$ for $\al>0$}, it is clear that
\begin{align*}
&\int_0^t\|\pi_\e (\cdot,s)\|_{L^{3/2}(B_\la)}^{3/2}	 \,ds 
+ \int_0^t \int 2\pi_\e (v_\e \cdot \nabla \phi) \,dx\,ds 
\\ &\qquad  \leq  C( \la,\ga,\eta) \int_0^t \big( \tilde \al_\e(s)^3+\tilde \al_\e(s)^1 \big)\,ds
+ \ga\int_0^t\int |\nabla v_\e|^2\phi\,dx\,ds.
\end{align*}


Combining the above estimates (and taking $\ga$ sufficiently small to absorb the gradient terms on the right hand side), we obtain
\begin{align}\label{ineq.alpha1}
\al_\e(t) +\int_0^t\int_{B_1}|\nabla v_\e|^2\,dx\,ds &\leq \al_0 +    C( \la,\eta,\ga) \int_0^t \big( \tilde \al_\e(s)^{3}  +\tilde \al_\e(s)^{1}  \big)\,ds.
\end{align}
Therefore,
\begin{align}\label{ineq.alpha2}
\tilde \al_\e(t)\leq \al_0 +    C  \int_0^t \big( \tilde \al_\e(s)^{3} +\tilde \al_\e(s)^{1}  \big)\,ds.
\end{align}
 By continuity of $\al_\e(t)$, we have
\begin{align}\label{ineq.alpha4}
\tilde \al_\e(t)\leq 2\al_0,\quad \forall t<T
\end{align}
for some $T>0$. By a continuity argument, we may take $T= (C(2 + 8\al_0^2))^{-1}$.

Letting $\e \to 0$ yields
\[
(v,\chi_{B_1} v) (t) \le \liminf_{\e \to 0} (v_\e,\chi_{B_1} v_\e)_{L^2} (t)\leq 2\al_0,
\]  
for all $t\leq T$.
Note that \eqref{ineq.alpha1} gives uniform (in $\e$) control of 
\[
\int_0^T\int_{B_1} |\nabla v_\e|^2\,dx\,dt \leq C(\al_0,\la)
\]
for some constant $C(\al_0,\la)$.
From \cite{BT1} we have that $v_\e$ converges {weakly} to $v$ in $L^2(\frac 1 k,T;H^1(B_1))$ for every $k\in \N$.  Hence, 
\[
\int_{1/k}^T\int_{B_1} |\nabla v |^2\,dx\,dt \leq \sup_{\e>0} \int_{0}^T\int_{B_1} |\nabla v_\e|^2\,dx\,dt,
\]
and, letting $k\to\I$, it follows that
\[
\int_{0}^T\int_{B_1} |\nabla v |^2\,dx\,dt
\leq C(\al_0,\la).
\]
Similarly, since $\pi_\e\in L^{3/2}(0,T; L^{3/2} (B_1))$ with uniformly bounded norms, it follows that $\pi\in L^{3/2}(0,T; L^{3/2} (B_1))$.
Applying Lemma \ref{lemma.pressure} yields the desired pressure representation in $L^{3/2}(0,T; L^{3/2} (B_1))$ and concludes the proof.
\end{proof}

\section{DSS solutions with data in $L^2_{\loc}(\R^3)$}\label{sec.construction}

In this section we prove Theorem \ref{thrm.main}.  To do this, we need to approximate DSS data in $L^2_{\loc}$ by divergence free DSS vector fields in $L^3_w$ and also characterize discrete self-similarity on $\R^3\times (0,\I)$ in terms of a neighborhood of the origin. 

\subsection{Approximation of DSS data in $L^2_{\loc}$}\label{sec.approx}

\begin{lemma}\label{lemma.approx}
Let $f\in L^2_{\loc}(\R^3;\R^3)$ be a given divergence free $\la$-DSS vector field for some $\la>0$.  There exists a sequence of divergence free $\la$-DSS vector fields $\phi^{(k)}$ so that $\phi^{(k)}\in L^3_w(\R^3)$ and $\|\phi^{(k)} -f\|_{L^2(B_1)} \to 0$ a $k\to \I$ ($B_1$ is the ball of radius $1$ centered at the origin).
\end{lemma} 
 
The main difficulty in proving this lemma is that each $f^{(k)}$ must be divergence free. We thus need to use the Bogovskii map which we presently recall, see 
\cite{MR631691}.

\begin{lemma}\label{lemma.bogovskii}Let $\Om$ be a bounded Lipschitz domain in $\R^n$, $2\le n <
	\infty$. There is a linear map $\Psi$ that maps a scalar $f \in L^q(\Om)$ 
	with $\int_\Om f = 0$, $1<q<\infty$, to a vector field $v=\Psi f \in W^{1,q}_0(\Om;\R^n)$
	and
	\begin{equation*}
	\div v= f, \quad \norm{v}_{W^{1,q}_0( \Om)} \le c(\Om,q) \norm{f}_{L^q( \Om)}.
	\end{equation*}
	The map $\Psi$ is independent of $q$ for $f \in C_c^\infty( \Om)$. 
\end{lemma}



\begin{proof}[Proof of Lemma \ref{lemma.approx}]

Let $Z_0(x)$ be $C^\I(\R^3)$ satisfy 
\[Z_0(x)=
\begin{cases}
1 & |x|>1   
\\\mbox{radial, increasing}& \la^{-1}\leq |x|\leq 1
\\0 & |x|<\la^{-1}
\end{cases}.
\]
Note that $\nabla\cdot (Z_0 f)=f\cdot \nabla Z_0$ -- i.e.~$Z_0 f$ is not divergence free.  We can correct this using Lemma \ref{lemma.bogovskii} with $q=2$ for the scalar $-f\cdot \nabla Z_0$ noting that $f$ is locally square integrable and 
\[
\int -f\cdot \nabla Z_0\,dx =0,
\]
because $f$ is divergence free. Denote by $\Phi_0$ the image of $-f\cdot \nabla Z_0$ under a Bogovskii mapping with domain $\{x:\la^{-1}\leq|x| \leq 1 \}$.  Then, {$\Phi_0 \in W^{1,2}_0(B_1\setminus B_{\la^{-1}})$ and}
\[
\nabla\cdot ( Z_0f + \Phi_0) =0.
\]
Let $Z_i(x)=Z_0(x/\la^i)$ and $\Phi_i(x) = \la^{-i}\Phi_0(\la^{-i}x)$ for all $i\in \Z$.  It follows that 
\[
\nabla\cdot(	Z_if+\Phi_i) = 0,
\]
for all $i\in \Z$. 
 Note that $\operatorname{supp}(Z_{j}-Z_{j+2}) =\{x: \la^{j-1}\leq |x| \leq \la^{j+2} \} $. 
Let
\[
{f_i  = \frac 12 (Z_{i} - Z_{i+2}) f  +\frac 12 \bke{\Phi_{i} - \Phi_{i+2}} .}
\]
Then each $f_i$ is divergence free and supported on $B_{\la^{i+2}}\setminus B_{\la^{i-1}}$. Furthermore, 
\[
{
f=\textstyle \sum_{i\in \Z} f_i
}
\] 
where convergence is understood in the point-wise sense for all $x\neq 0$.  To confirm this note that, if $x$ satisfies $\la^i\leq|x|<\la^{i+1}$ then $x\in\operatorname{supp} (Z_{j}-Z_{j+2})$ if and only if $j\in \{ i-1,i,i+1\}$.  It follows that $\sum_{j\in \Z} (Z_j -Z_{j+2})(x) = 2$.  On the other hand, $\operatorname{supp} \Phi_j=\{x: \la^{j-1} \leq|x|\leq \la^{j} \} $ and, therefore, $\sum_{j\in\Z}( \Phi_j(x)-\Phi_{j+2}(x) ) = \Phi_{i+1}(x)-\Phi_{i+1}(x)=0$.  It follows that $f=\sum_{i\in \Z} f_i$.

Assume $\phi_0^{(k)}$ is a sequence of divergence free vector fields in $C_0^\I(B_{\la^{2}}\setminus B_{\la^{-1}} )$ so that $\phi_0^{(k)}\to f_0$ in $L^2(B_{\la^{2}}\setminus B_{\la^{-1}} )$. Let $\phi_i^{(k)} = \la^{-i}\phi_0^{(k)}(\la^{-i}x)$.  Then the vector field
\[
\phi^{(k)}=\sum_{i\in \Z} \phi_i^{(k)},
\]  
is a divergence free, $\la$-DSS vector field, and satisfies
\[
|\phi^{(k)}(x)|\le c_k |x|^{-1},
\]
(where the proportionality constants $c_k$ are \emph{not} uniformly bounded with respect to $k$).  Hence, $\phi^{(k)}\in L^3_w$. We finish by arguing that $\phi^{(k)}\to f$ in $L^2(B_1)$. We know that $\int_{B_{\la^2}\setminus B_{\la^{-1}}} (\phi_0^{(k)}-f )^2\,dx \to 0$ as $k\to \I$.  Using the definition of $\phi^{(k)}$ and the fact that $f$ is discretely self-similar we have, 
 {letting $A_i = B_{\la^i} \setminus B_{\la^{i-1}}$, that
\begin{align*}
\int_{B_1} (\phi^{(k)}-f)^2\,dx &=   \sum_{i\leq 0} \int_{A_i} (\phi^{(k)}-f)^2\,dx
\\& =   \sum_{i\leq 0} \la^i \int_{A_0} (\phi^{(k)}-f)^2\,dx  = \frac\la{\la-1}\int_{A_0} (\phi^{(k)}-f)^2\,dx.
\end{align*}
In $A_0$, we have $\phi^{(k)}-f =\sum_{i=-2}^0 ( \phi^{(k)}_i-f_i)$. Thus
\EQN{
\norm{\phi^{(k)}-f}_{L^2(A_0)} &\le \sum_{i=-2}^0 \norm{ \phi^{(k)}_i-f_i}_{L^2(A_0)}
\\
& =  \sum_{k=0}^2 \la^{-k/2} \norm{ \phi^{(k)}_0-f_0}_{L^2(A_k)}
\le 3 \norm{ \phi^{(k)}_0-f_0}_ {L^2(B_{\la^2} \setminus B_{\la^{-1}})},
}
}
which completes the proof.
\end{proof}

\subsection{DSS solutions in a neighborhood of the origin}\label{sec.extension}
 
 In the introduction we saw that any time-periodic solution $u$ to \eqref{eq:timeDependentLeray} corresponds to a DSS solution $v$ after the change of variables \eqref{variables}.  Distributionally, $u$ is a time-periodic solution to \eqref{eq:timeDependentLeray} if and only if 
 \begin{equation}\label{u.eq-weak}
 \int_{s'}^{s'+T} \big( (u,\partial_s f)-(\nabla u,\nabla f)+
 ( {\frac 12 u+\frac12 y\cdot\nabla u } -u\cdot\nabla u ,f)  \big)  \,ds =0,
 \end{equation} 
 holds for all $s'\in \R$ and $f \in \mathcal D_T$ where  $\mathcal D_T$ denotes the collection of all smooth divergence free vector fields in $\R^3 \times \R$ which 
 are time periodic with period $T$ and whose supports are compact in space.  In \cite{BT1}, this definition was used with $s'=0$ since the goal was to extend a solution on $[0,T]$ to $\R$ using periodicity.  The same modification can be made here based on the observations that if $u$ satisfies \eqref{u.eq-weak} 
then $u$ can be extended to a time-periodic solution on $\R$ and if $u$ is a time-periodic solution on $\R$ then $u$ satisfies \eqref{u.eq-weak}. 
 
 Since there is a one to one correspondence between time-periodic solutions to \eqref{eq:timeDependentLeray} and DSS solutions, an equivalent characterization of DSS solutions is obtained by reformulating \eqref{u.eq-weak} in the physical variables. For $f\in \mathcal D_T$ let $\zeta_f(x,t)=t^{-1}f(y,s)$. {Note $\zeta_f(x,t) = \la^2 \zeta_f(\la x,\la^2 t)$. }
  Then, $v$ is $\la$-DSS if and only if
  \begin{equation}\label{v.eq-weak}
  {\int_t^{\la^2t} }\big( (v,\partial_t \zeta_f)-(\nabla v,\nabla \zeta_f)-(v\cdot\nabla v ,\zeta_f)  \big)  \,d\tau =0, 
  \end{equation}
  for all $t>0$ and $f\in \mathcal D_T$, since \eqref{u.eq-weak} is {just \eqref{v.eq-weak} in} similarity variables. 
  {Note that $(v,\zeta_f)|_{\tau=\la^2t} = (v,\zeta_f)|_{\tau=t}$.}
 It follows that, if $v$ is a solution to \eqref{eq.NSE} that satisfies \eqref{v.eq-weak} for $t=1$, then $v|_{\tau\in [1,\la^2]}$ can be extended to a $\la$-DSS solution for all positive times.

 Fix $k\in \Z$ and let $Q_k = B_{\la^k}(0)\times(0,\la^{2k})$.  Our goal is to give a third characterization of discrete self-similarity on $Q_k$.  Let $f\in \mathcal D_T$ be given and $\zeta_f$ be as above.  Let $R$ be large enough so that, for all $t\in [1,\la^2]$, the support of $\zeta_f(t)$ is a subset of $B_R(0)$ and choose $m=m(f)\in \Z$ so that $R/\la^m<\la^k$ and $\la^{2-2m}<\la^{2k}$.  It follows that
 \[
 	B_{R/\la^m}(0)\times [\la^{-2m},\la^{2-2m}] \subset Q_k.
 \]
Extend $\zeta_f$ to all $t>0$ using the following scaling: For $(x,t)\in \R^3\times (0,\I)$, let 
 \[
 \zeta_f(x,t)= \la^{2i}\zeta_f(\la^ix,\la^{2i}t),
 \] 
 where $i$ is chosen so that $\la^{2i}t\in [1,\la^2]$.
Since $\zeta_f|_{\R^3\times [1,\la^2]}$ is compactly supported in space, it's spatial support shrinks as $t\to 0^+$.  In particular, for $t\in [\la^{-2m},\la^{2-2m}]$, $\supp \zeta_f \subset Q_k$.  For $m\in \Z$, let 
\EQ{\mathcal D_{Q_k}^m=  \{   & \phi\in C^\I(\R^3\times (0,\I)): \supp \phi|_{t\in [\la^{-2m},\la^{2-2m}]} \subset Q_k  
	\\& \text{ and } \forall\, (x,t)\in \R^3\times (0,\I), \exists\,    f\in \mathcal D_T
 \text{ such that } \phi(x,t)=\zeta_f(x,t).   
\} }
It is easy to see that, $\cup_{m\in \Z}  \mathcal D_{Q_k}^m = \mathcal D_T$.

Re-scaling \eqref{v.eq-weak} gives 
 \begin{equation}\label{v.eq-core}
  \int_{\la^{-2m}}^{\la^{2-2m}} \big( (v,\partial_t \zeta_f)-(\nabla v,\nabla \zeta_f)-(v\cdot\nabla v ,\zeta_f)  \big)  \,dt' =0, 
  \end{equation}
  where $t'=t/\la^{2m}$ and the innerproducts are taken with respect to the re-scaled spatial variable $x'=x/\la^m$.  In particular, the integral is computed over a subset of $Q_k$ and is identical to the same integral with $\zeta_f$ replaced by $\phi$ for some $\phi\in \mathcal D^m_{Q_m}$.    Thus, if $v$ is a solution to \eqref{eq.NSE}, and $\phi\in \mathcal D_{Q_k}^m$ for some $m\in \Z$, then \eqref{v.eq-weak} is satisfied if and only if \eqref{v.eq-core} is satisfied for the $f\in \mathcal D_T$ for which $\zeta_f=\phi$.  This leads to the following extendability property: If $v$ is a solution to \eqref{eq.NSE} on $Q_k$ and satisfies \eqref{v.eq-core} for every $m\in \Z$ and $\phi \in \mathcal D_{Q_k}^m$,
   then $v$ can be extended to a discretely self-similar solution on $\R^3\times (0,\I)$; in other words, if a solution is DSS in a neighborhood of the origin, then it can be extended to a DSS solution on $\R^3\times (0,\I)$.

\subsection{Construction of DSS solutions}\label{sec.construction}

\begin{proof}[Proof of Theorem \ref{thrm.main}]  Fix $\la > 1$ and assume $v_0\in L^2_{\loc}$ is a divergence free $\la$-DSS vector field. 
Let $\{ v_0^{(k)} \}$ be the sequence of vector fields $\{\phi^{(k)}\}$ from Lemma \ref{lemma.approx} applied to $v_0$.  Then, the values $\|v_0^{(k)}\|_{L^2(B_1)}$ are uniformly bounded and $\|v_0^{(k)}-v_0\|_{L^2(B_1)}\to 0$ as $k\to \I$. 
Since $v_0^{(k)}\in L^3_w$ and is $\la$-DSS, by \cite{BT1} there exists a $\la$-DSS local Leray solution $v_k$ to \eqref{eq.NSE} and an associated pressure $\pi_k$ having initial data $v_0^{(k)}$ for every $k\in \N$. 
By Proposition \ref{lemma.energy}, $v_k$ are uniformly bounded in $L^\I(0,T;L^2(B_1))\cap L^2(0,T;H^1(B_1))$ {(hence also in $L^{10/3}(0,T;L^{10/3}(B_1))$)} for some $T$ which depends only on $\la$ and $\| v_0^{(k)} \|_{L^2(B_1)}$.   As usual (cf. \cite{BT1,KiSe,LR2}), there exists a distribution $v$ and a subsequence of $\{v_k\}$ (still indexed by $k$ for simplicity) so that $v_k$ converges to $v$ in the weak star topology on $L^\I(0,T;L^2(B_1))$, in the weak topology on $L^2(0,T;H^1(B_1))$, and in $L^2(0,T;L^2(B_1))$. 
{Since they are uniformly bounded in $L^{10/3}(0,T;L^{10/3}(B_1))$, they also converge in $L^q(0,T;L^q(B_1))$ for any $q<10/3$}. 
By the pressure estimate 
 {\eqref{prop3.1-2} in Proposition \ref{lemma.energy}, 
$\pi_k$ are uniformly bounded in $L^{3/2}(0,T;L^{3/2}(B_1))$ 
by $C(\la, \| v_0 \|_{L^2(B_\la)})$}
and, therefore, we may extract a subsequence which converges weakly to a distribution $\pi \in L^{3/2}(0,T;L^{3/2}(B_1))$.

Fix  $\kappa\in \Z$  so that $\la^\kappa <1$ and $\la^{2\kappa}<T$.  Then, $Q_\kappa
=B_{\la^\kappa} \times (0,\la^{2\kappa})
 \subset B_1\times (0,T)$. Therefore $v_k$ satisfies \eqref{eq.NSE} on $Q_\kappa$  and satisfies \eqref{v.eq-core} for every $m\in \Z$ and $\phi\in \mathcal D^m_{Q_\kappa}$.    
 Thus, $v$ can be extended to a DSS solution on $\R^3\times (0,\I)$ (which we still denote by $v$). 

For compact subsets $K$ of $B_1$, we automatically have $\lim_{t\to 0^+} \| v-v_0\|_{L^2(K)}=0$.  For a general compact subset $K$ of $\R^3$, we have $K'=\la^{m} K \subset B_1$ for some $m \in \Z$, and
\[
\int_K |v(x,t)- v_0(x)|^2\,dx =\la^{-m} \int_{K'} |v(x',\la^{2m}t) -v_0(x') |^2\,dx'.
\]
It follows that $\lim_{t\to 0^+} \| v(t)-v_0\|_{L^2(K)}=0$ for every compact set $K\subset \R^3$. A similar re-scaling argument also implies that $v\in L^\I(0,T';L^2(K))\cap L^2(0,T';H^1(K))$   and $\pi \in L^{3/2}(0,T';L^{3/2}(K))$ for any $T'>0$ and compact subset $K$ of $\R^3$.

To confirm that $v$ satisfies the local energy inequality, first note that each $v_k$ satisfies the local energy inequality
\begin{align*}&\int |v_k(t)|^2\phi \,dx +2\int \int |\nabla v_k|^2\phi\,dx\,dt 
\\&\leq\int |v_0^{(k)}|^2\phi \,dx+ \int\int |v_k|^2(\partial_t \phi + \Delta\phi )\,dx\,dt +\int\int (|v_k|^2+2\pi_k)(v_k\cdot \nabla\phi)\,dx\,dt,
\end{align*}
for all non-negative $\phi \in C_0^\I(\R^3\times \R^3_+)$.
Furthermore, the right hand sides of the energy inequality for $v^{(k)}$ converge to the right hand side of the energy inequality for $v$ as $k\to\I$ while the left hand-sides are lower semi-continuous (cf. \cite[(A.51)]{CKN}). The local energy inequality for $v$ plainly follows.

Finally, note that $\pi_k$ satisfies the formula \eqref{prop3.1-3}. Applying Lemma \ref{lemma.pressure} to the above sequence and limit implies that $\pi$ satisfies the desired pressure formula in $L^{3/2}(0,T; L^{3/2} (B_1))$.  Re-scaling establishes the formula in $L^{3/2}_{\loc}(\R^3\times (0,\I))$.
\end{proof}

\section*{Acknowledgments}
The research of Tsai was partially supported by the NSERC grant 261356-13 (Canada).

Zachary Bradshaw, Department of Mathematics, University of Arkansas, Fayetteville, AR 72701, USA;
e-mail: zb002@uark.edu
\medskip

Tai-Peng Tsai, Department of Mathematics, University of British
Columbia, Vancouver, BC V6T 1Z2, Canada;
e-mail: ttsai@math.ubc.ca

\end{document}